\pdfoutput=1
\RequirePackage{ifpdf}
\ifpdf 
\documentclass[pdftex]{sigma}
\else
\documentclass{sigma}
\fi

\numberwithin{equation}{section}

\newtheorem{Theorem}{Theorem}[section]
\newtheorem*{Theorem*}{Theorem}
\newtheorem{Corollary}[Theorem]{Corollary}
\newtheorem{Lemma}[Theorem]{Lemma}
\newtheorem{Proposition}[Theorem]{Proposition}
 { \theoremstyle{definition}
\newtheorem{Definition}[Theorem]{Definition}

\newtheorem{Example}[Theorem]{Example}
\newtheorem{Remark}[Theorem]{Remark} }

\DeclareMathOperator{\End}{End}
\DeclareMathOperator{\im}{im}

\DeclareMathOperator{\Hom}{Hom}

\DeclareMathOperator{\Tr}{Tr}

\newcommand{\fg}{\mathfrak{g}}
\newcommand{\fh}{\mathfrak{h}}

\newcommand{\fl}{\mathfrak{l}}

\newcommand{\fo}{\mathfrak{o}}

\newcommand{\fs}{\mathfrak{s}}

\newcommand{\fD}{\mathfrak{D}}

\newcommand{\bbc}{\mathbb{C}}

\newcommand{\bbr}{\mathbb{R}}

\newcommand{\bbt}{\mathbb{T}}

\newcommand{\bbz}{\mathbb{Z}}

\newcommand{\Cc}{\mathcal{C}}

\newcommand{\Cu}{\mathcal{U}}

\newcommand{\Cw}{\mathcal{W}}

\newcommand{\bx}{\mathbf{x}}
\newcommand{\by}{\mathbf{y}}

\newcommand{\bM}{\mathbf{M}}

\newcommand{\ux}{\underline{x}}

\newcommand{\uD}{\underline{D}}

\newcommand{\uM}{\underline{M}}

\newcommand{\lpi}{\langle}
\newcommand{\rpi}{\rangle}

\newcommand{\Pin}{\mathsf{Pin}}

\DeclareMathOperator{\rca}{\mathsf{H}}
\DeclareMathOperator{\ama}{\mathsf{A}}

\begin{document}
\allowdisplaybreaks

\newcommand{\arXivNumber}{2110.01353}

\renewcommand{\PaperNumber}{040}

\FirstPageHeading

\ShortArticleName{Dirac Operators for the Dunkl Angular Momentum Algebra}

\ArticleName{Dirac Operators for the Dunkl Angular Momentum\\ Algebra}

\Author{Kieran CALVERT~$^{\rm a}$ and Marcelo DE MARTINO~$^{\rm b}$}

\AuthorNameForHeading{K.~Calvert and M.~De Martino}

\Address{$^{\rm a)}$~Department of Mathematics, University of Manchester, UK}
\EmailD{\href{mailto:kieran.calvert@manchester.ac.uk}{kieran.calvert@manchester.ac.uk}}

\Address{$^{\rm b)}$~Department of Electronics and Information Systems, University of Ghent, Belgium}
\EmailD{\href{mailto:marcelo.goncalvesdemartino@ugent.be}{marcelo.goncalvesdemartino@ugent.be}}

\ArticleDates{Received November 10, 2021, in final form May 24, 2022; Published online June 01, 2022}

\Abstract{We define a family of Dirac operators for the Dunkl angular momentum algebra depending on certain central elements of the group algebra of the Pin cover of the Weyl group inherent to the rational Cherednik algebra. We prove an analogue of Vogan's conjecture for this family of operators and use this to show that the Dirac cohomology, when non-zero, determines the central character of representations of the angular momentum algebra. Furthermore, interpreting this algebra in the framework of (deformed) Howe dualities, we show that the natural Dirac element we define yields, up to scalars, a square root of the angular part of the Calogero--Moser Hamiltonian.}


\Classification{16S37; 17B99; 20F55; 81R12}

\section{Introduction}
Let $(E,B)$ be a Euclidean space and consider the action of partial differential operators with polynomial coefficients in the space of polynomial functions $\bbc[E]$. This framework is very fruitful and yields many applications most importantly to physics. Angular momentum, for instance, is a fundamental property of particle dynamics and the quantum angular momentum operators are realized within this setup. We consider the situation in which the partial differential operators are deformed to differential-difference operators, the so-called Dunkl operators. For this, we also need a real reflection group~$W$ inside the orthogonal group $\mathsf{O}(E,B)$ and a parameter function~$c$ on the conjugacy classes of reflections of~$W$. Together, the pair~($W,c)$, the Dunkl operators and the multiplication operators generate the so-called rational Cherednik algebra (see Definition~\ref{d:RCA}) inside the endomorphism space of the polynomial ring~$\mathbb{C}[E]$.

The subalgebra of the Cherednik algebra generated by $W$ and the Dunkl angular momentum operators is called the Dunkl angular momentum algebra (see Definition~\ref{d:AMA}). In \cite{FH15}, Feigin and Hakobyan obtained important structural results about this algebra. In particular they obtained all the defining relations and showed that its centre is, essentially, a univariate polynomial ring on the angular part of the Calogero--Moser Hamiltonian (see also \cite[Remark~3.3]{FH19}). Later in \cite{CDM20}, it was shown that this algebra naturally arises in the context of deformed Howe dualities as the centralizer algebra of the Dunkl--Cherednik version of the polynomial $\mathfrak{sl}(2)$-triple obtained from the Laplacian and the norm-squared operator. It is then clear that the angular Calogero--Moser Hamiltonian is, up to scalars, the Casimir operator of $\mathfrak{sl}(2)$ (see Remark~\ref{r:AngHamiltonian}, below).

In this paper, inspired by the successful theory of Dirac operators for Lie theory \cite{AS77,HP02,K03,P72,V97}
and Drinfeld algebras \cite{BCT12,C20,C16,CDM18, COT14}, we propose to define a theory of Dirac operators for the Dunkl angular momentum algebra. In slightly more details, we work with the Clifford algebra associated to $(E,B)$ and we define the Dirac element~$\mathcal{D}$ inside the tensor product of the angular momentum algebra and the Clifford algebra. We then show that this element is invariant for~$\tilde{W}$, the Pin-cover of the Weyl group $W$, and that by a suitable modification~$\phi$ (see Definition~\ref{e:CDirac}), akin to the one made by Kostant~\cite{K03} in the context of cubic Dirac operators, the element $\mathfrak{D}_0 = \mathcal{D} - \phi$ is essentially a square-root of the Casimir of $\mathfrak{sl}(2)$ (see Corollary~\ref{c:DiracSquare}). Furthermore, we introduce a family of Dirac operators $D_C$ depending on certain central elements $C$ of $\bbc \tilde W$ (see Definition~\ref{d:Dirac_C}) with respect to which we prove an analogue of Vogan's conjecture (see Theorem~\ref{t:Vogan}) and, using the celebrated notion of Dirac cohomology (see Definition~\ref{d:DiracCoh}), we show that the Dirac cohomology, when non-zero, determines the central character of representations of the angular momentum algebra (see Theorem~\ref{t:CentChar}). We expect that such results can aid in a systematic study of the representation theory of the angular momentum algebra, since its representation theory, just like for the rational Cherednik algebra, is highly dependant on the parameter function $c$.

Finally, we give a break-down of the contents of the paper. In Section \ref{s:Prelims}, we recall the definition of the rational Cherednik algebra, introduce the angular momentum algebra and obtain a linear relation between the Casimir of $\mathfrak{sl}(2)$ and the angular Calogero--Moser Hamiltonian. Next, in Section~\ref{s:Cliff} we recall the definitions of the Clifford algebra, the Pin-cover of the Weyl group and we introduce the Dirac elements of the angular momentum algebra. The highlight of this section is the computation of the square. Afterwards, in Section \ref{s:SCasimir} we relate our Dirac element $\mathcal{D}$ with the SCasimir of the closely related algebra $\mathfrak{osp}(1|2)$ (see De Bie et al.~\cite{DOV18a,DOV18b} for the explicit realization) while in Section~\ref{s:Cohomology}, we prove the main results on Vogan's conjecture and Dirac cohomology. In the last section, we describe and study a non-trivial example of an admissible central element that yields a Dirac operator and relate such element with the Dirac operator obtained by~\cite{BCT12}, in the context of a graded affine Hecke algebra. We also discuss the set of admssible elements in the case when $W=S_n$.

\section{Preliminaries}\label{s:Prelims}

Let $(E,B)$ be a Euclidean space affording the reflection representation of a finite reflection group $W\subset \mathsf{O}(E,B)$. Put $n = \dim(E)$. Let $R\subseteq E^*$ denote the root system of $W$ and $R^\vee\subseteq E$ its dual root system normalized by the condition $\langle \alpha,\alpha^\vee \rangle = 2$, for all $\alpha$ in $R$, where $\langle -,- \rangle\colon E^*\times E \to \bbr$ denotes the natural pairing. We shall identify $E$ and $E^*$ isometrically using the Euclidean structure $B$ and we denote by $B^*$ the inherited Euclidean structure on $E^*$. This identification $B\colon E \to E^*$ is defined by $\langle B(y),\eta \rangle = B(y,\eta)$ for all $y,\eta\in E$.
\begin{Remark}
Under the isometry $B\colon E \to E^*$, we have
$\alpha = 2B(\alpha^\vee)|\alpha^\vee|^{-2}$ and $2=|\alpha||\alpha^\vee|$. Further, if $\{y_1,\dots,y_n\}\subset E$ is an orthornormal basis then $\{x_1,\dots,x_n\}\subset E^*$ is an orthonormal basis, where $x_i = B(y_i)$ for all $i$, and the pairings are related via
\begin{equation}\label{e:pairings}
\langle x_i,\alpha^\vee \rangle = B(y_i,\alpha^\vee)
= \frac{2}{|\alpha|^2} B^*(x_i,\alpha)
= \frac{2}{|\alpha|^2}\langle \alpha,y_i \rangle
= \frac{|\alpha^\vee|^2}{2}\langle \alpha,y_i \rangle,
\end{equation}
for all $1\leq i \leq n$.
\end{Remark}
Fix, once and for all, a positive system $R_+\subseteq R$ and let $c\colon R_+\to \bbc$ be a parameter function, that is, an assignment $\alpha \mapsto c_\alpha \in \bbc$ such that $c_\alpha = c_{w\alpha}$ for all $w \in W$. Let $\Delta$ be the simple roots determined by $R_+$. Denote by $\fh = E_\bbc$ and $\fh^* = E^*_\bbc$. For any $\alpha \in R$, the element $s_\alpha$ is the reflection in $W$ acting by
\[ s_\alpha(y) = y - \langle \alpha , y \rangle \alpha^\vee,
\]
for all $y \in E$.

\begin{Definition}[\cite{EG02}] \label{d:RCA}
The {\it rational Cherednik algebra} $\rca = \rca(\fh,W,c)$ is the quotient of the smash product algebra $\bbt(\fh^*\oplus\fh)\#W$ modulo the relations $[x,x']=0=[y,y']$ and
\begin{equation*}
[y,x] = \langle x,y \rangle + \sum_{\alpha>0} c_\alpha \langle \alpha,y\rangle \langle x, \alpha^\vee \rangle s_\alpha,
\end{equation*}
for all $y,y'\in \fh$ and $x,x' \in \fh^*$.
\end{Definition}

\begin{Remark}
More generally, rational Cherednik algebras are defined with respect to finite complex reflection groups inside the unitary group with respect to the Hermitian extension of~$B$. However, for the existence of the $\mathfrak{sl}(2)$-triple and the Duality Theorem stated below, it is fundamental that $W$ is a real reflection group.
\end{Remark}

Fix an orthonormal basis $\{y_1,\dots, y_n\}\subset E$ and let $\{x_1,\dots,x_n\}\subset E^*$ be the dual basis, i.e., with $x_i = B(y_i)$ for all $i$. Consider the vector notation $\bx:=(x_1,\dots,x_n)$ and $\by := (y_1,\dots,y_n)$ with the usual dot product of vectors. As customary, we shall write $\bx^2$ for $\bx\cdot\bx$ and similarly for~$\by^2$. It is well-known (see \cite{He91}) that the elements $H := \tfrac{1}{2}(\bx\cdot\by + \by\cdot\bx), X:= -\tfrac{1}{2}\bx^2$ and $Y:=\tfrac{1}{2} \by^2$ of $\rca$ satisfy the $\fs\fl(2)$-commutation relations and span a copy of $\fs\fl(2,\bbc)$ inside $\rca$. On the other hand, consider the Dunkl angular momentum elements $M_{ij} := x_iy_j - x_jy_i$ of $\rca$ for $1\leq i,j\leq n$. Note that they span a vector space isomorphic to $\wedge^2(\fh)$. For each pair $(i,j)$ with $1\leq i, j\leq n$ define
\[
S_{ij} := [y_i,x_j] = \delta_{ij} + \sum_{\alpha >0} c_\alpha \langle \alpha,y_j \rangle \langle x_i, \alpha^\vee \rangle s_\alpha\in\bbc W
\]
and let $Z := \sum_{\alpha>0} c_\alpha s_\alpha$. Note that $Z$ is in the centre of $\bbc W$ since the parameter function $c$ is uniform on conjugacy classes of reflections. Since $W$ is a real reflection group, we get $S_{ij}=S_{ji}$, for all $i$, $j$.

\begin{Lemma}
We have
\begin{equation}\label{e:commutatorsum}
 \sum_i S_{ii} = \sum_i [y_i,x_i] = n + 2Z.
\end{equation}
\end{Lemma}

\begin{proof}
Using $S_{ii} = [y_i,x_i] = 1+\sum_{\alpha >0} c_\alpha \langle \alpha ,y_i\rangle \langle x_i,\alpha^\vee \rangle s_\alpha$ and the identity
$
\sum_i \langle \alpha,y_i \rangle \langle x_i,\alpha^\vee \rangle = \langle \alpha, \alpha^\vee\rangle =2,
$
the claim follows.
\end{proof}

\begin{Definition}[\cite{FH15}]\label{d:AMA}
Let $\{\uM_{ij}\mid 1\leq i<j\leq n\}$ be a vector space basis of $\wedge^2(\fh)$. The {\it Dunkl angular momentum algebra} $\ama(\fh,W,c)$ is the quotient of the smash product algebra $\bbt\big({\wedge}^2(\fh)\big)\#W$ modulo the commutation relations
\begin{equation}\label{e:agacomrel}
[\uM_{ij},\uM_{kl}] = \uM_{il}S_{jk} + \uM_{jk}S_{il}-\uM_{ik}S_{jl} - \uM_{jl}S_{ik}
\end{equation}
and the crossing-relations
\begin{equation}\label{e:agacrossrel}
\uM_{ij}\uM_{kl} + \uM_{jk}\uM_{il}+\uM_{ki}\uM_{jl} = \uM_{ij}S_{kl} + \uM_{jk}S_{il}+\uM_{ki}S_{jl}
\end{equation}
for all $1\leq i,j,k,l \leq n$.
Note that $\uM_{ii} = 0$ for any $i = 1,\dots, n$.
\end{Definition}

In what follows, we shall refer to this algebra only as the angular momentum algebra, or just AMA. The relevance of this subalgebra of $\rca$ is manifested by the following fact (see~\cite{FH15} and~\cite{CDM20}):

\begin{Theorem}\label{t:Centre}
The associative subalgebra $\ama$ of $\rca$ generated by the elements $\{M_{ij}\mid 1\leq i<j \leq n\}$ and $W$ is isomorphic to the angular momentum algebra $\ama(\fh,W,c)$. Furthermore, $\ama$ is the centralizer algebra in $\rca$ of the $\fs\fl(2)$-triple $(H,X,Y)$.
\end{Theorem}

Let $w_0$ be the longest element of $W$ with respect to the simple roots $\Delta$. Then, $-w_0$ acts on the root system $R$ and it is an automorphism of the associated Dynkin diagram.

\begin{Definition}
We will denote by $(-1)_\fh$ the element $-{\rm Id}\in\End(\fh)$.
\end{Definition}

\begin{Remark}
If we have $w_0=(-1)_\fh$, then $w_0$ is in the centre of $W$ and acts on $\mathfrak{h}$ and $\mathfrak{h}^*$
 by~$-1$ and hence trivially on $\wedge^2\mathfrak{h}$.
\end{Remark}

\begin{Lemma}\label{r:schurorth}
The only elements of $W$ which act trivially on $\wedge^2\mathfrak{h}$ are, respectively, $1_W$ and $(-1)_\fh$ if $w_0 = (-1)_\fh$, and only $1_W$ if $w_0\neq (-1)_\fh$.
\end{Lemma}

\begin{proof} Note that, by Schur's lemma, the only elements of the orthogonal group acting trivially on
$\wedge^2\mathfrak{h}$ are $\pm {\rm Id}$. The statement follows from this observation.
\end{proof}

Since $(H,X,Y)$ span a Lie algebra isomorphic to $\fs\fl(2,\bbc)$, the associative algebra subalgebra of $\rca$ generated by this triple contains the quadratic Casimir element $\Omega_{\fs\fl(2)} := H^2 +2(XY+YX)$. The centre of $\ama$ is given in terms of $\Omega_{\fs\fl(2)}$ and, possibly, $(-1)_\fh$.

\begin{Lemma}
When $c=0$ and $W$ is the trivial group, the center of $\ama$ is the univariate polynomial ring $\bbc[\Omega_{\fs\fl(2)}]$.
\end{Lemma}

\begin{proof}
This statement is a consequence of classical invariant theory (see \cite{H89} and \cite{W46}), but we provide the argument for completeness. In the present situation, $\ama$ is a subalgebra of the Weyl algebra $\Cw$ acting in the space of polynomial functions $\bbc[E]$. Furthermore, if we denote by $G$ the orthogonal group, $G_0$ its identity component (the special orthogonal group), $\fg=\textup{Lie}(G)$ and $\Cu(\fg)$ the universal enveloping algebra, then there is a $G$-equivariant homomorphism $\varphi\colon \Cu(\fg)\to \Cw$, whose image coincides with $\ama$. By equivariance, it follows that $\varphi$ maps $\Cu(\fg)^G\to\Cw^G$.

All that said,
if $Z$ is in the center of $\ama$, then $Z$ commutes with every generator $M_{ij}$ of $\ama$ and hence $Z\in \Cw^{G_0}$. Further, as $Z$ is also in the image of $\varphi$, it follows that $Z = \varphi(\tilde{Z})$ for some $\tilde{Z} \in \Cu(\fg)^{G_0} = \Cu(\fg)^{G}$, which then implies that $Z\in \Cw^G$. By classical invariant theory, $Z$ is thus in the associative subalgebra of $\Cw$ generated by the $\fs\fl(2)$-triple $(H,X,Y)$, from which the statement follows.
\end{proof}

\begin{Theorem}
The centre of $\ama$ is equal to the polynomial ring $\mathcal{R}[\Omega_{\fs\fl(2)}]$ on the Casimir with coefficients $\mathcal{R}= \mathbb{C}[(-1)_\fh]$, if $w_0=(-1)_\fh$, or $\mathcal{R}=\mathbb{C}$ otherwise.
\end{Theorem}

\begin{proof}
 The proof can be split in two cases; either $w_0=(-1)_\fh$ or not. In the later, the proof is identical to \cite[Theorem~5]{FH15}. For the remainder of this proof we assume that $w_0=(-1)_\fh$. Since the element $(-1)_\fh$ is in the centre of~$W$ and acts by~$1$ on $\wedge^2 \mathfrak{h}$ it is in the centre of $\ama$. We are left to prove that the subalgebra generated by $\Omega_{\fs\fl(2)}$, $(-1)_\fh$ and the constants is the full centre~$Z(\ama)$.

 Let $F$ be an arbitrary element in $Z(\ama)$. With respect to the usual filtration of $\rca$ whose associated graded object gives the PBW isomorphism $\rca = \bbc[\fh]\otimes\bbc[\fh^*]\otimes\bbc W$, let $F_0$ be the highest degree component occurring in $F$, say, of degree $d$. We can write $F_0 = \sum_{w\in W}p_w w$, where $p_w$ is a homogeneous polymonomial on the basis $\{x_1,\dots,x_n,y_1,\dots,y_n\}$ of degree $d$ and $w\in W$. We claim $p_w=0$ unless $w=1_W$ or $w=(-1)_\fh$. Suppose not and let $w \in W\setminus\{1_W,w_0\}$. By Lemma \ref{r:schurorth}, $w$ does not act trivially on $\wedge^2{\fh}$ and hence there exists an $M_{ij}$ such that $w(M_{ij}) \neq M_{ij}$, thus $[F,M_{ij}]$ has terms of degree bigger than $d$. However, as $F$ is in the centre of $\ama$ we get $[F,M_{ij}]=0$. This contradiction proves that the only group elements occurring in $F_0$ are $1_W$ or~$(-1)_\fh$.

 Furthermore, the top degree elements in $[F_0,M_{ij}]$ agrees with the top degree elements of $[F,M_{ij}]$. Therefore, modulo lower order terms $F_0$ is in $Z(\ama)$. Write $F_0 = p_{-1}(-1)_\fh+ p_11_W$. We claim that $p_{1}$ and $p_{-1}$ are polynomials over $\mathbb{C}$ in the variable $\Omega_{\fs\fl(2)}$. Note that the top degree elements of $[p_{-1},M_{ij}]$ and $[p_{1},M_{ij}]$ agree with the classical commutators (at $c=0$). The classical center ($c=0$) is generated by $\Omega_{\fs\fl(2)}$ and we can write $p_{-1}$ and $p_{1}$ as the corresponding elements in the classical center, modulo lower degree terms. We have thus proved that, modulo lower degree terms, $F_0$ is in the algebra
 $\mathbb{C}[(-1)_\fh][\Omega_{\fs\fl(2)}]$ and $F-F_0$ has lower degree. By induction $F$ is in $\mathcal{R}[\Omega_{\fs\fl(2)}]$ and we are done.
\end{proof}

\begin{Remark}
The above result is not novel. In \cite{FH15}, it was shown that for $W=S_n$, the centre of~$\ama$ is equal to the univariate polynomial ring on the angular Calogero--Moser Hamiltonian, which coincide with $\Omega_{\mathfrak{sl}(2)}$, modulo lower degree terms (see Remark~\ref{r:AngHamiltonian}, below). In \cite[Remark~3.3]{FH19}, the above theorem was stated, without proof, for general $W$. We decided to present the argument here for completeness.
\end{Remark}

Now let $\bM^2 := \sum_{i<j} M_{ij}^2\in\ama$ be the Dunkl angular momentum square. In what comes next, we shall compute the precise relationship between $\bM^2$ and the Casimir $\Omega_{\fs\fl(2)}$. Recall the central element $Z = \sum_{\alpha>0} c_\alpha s_\alpha$ of $\bbc W$.

\begin{Proposition}
The Dunkl angular momentum square satisfy the identity
\[
\bM^2 = \bx^2\by^2 - (\bx\cdot\by)^2 - (\bx\cdot\by)(2Z + n - 2).
\]
\end{Proposition}

\begin{proof} This is equation $(2.14)$ in \cite{FH15}, we add further details for completeness.
Define $Q := \sum_{i,j}x_ix_jy_iy_j$ and $\Sigma := \sum_{\alpha >0} c_\alpha\alpha\alpha^\vee s_\alpha$. Here, we see $\alpha\in\fh^*$ and $\alpha^\vee\in\fh$ as elements of $\rca$. Explictly, $\alpha = \sum_i \langle \alpha , y_i \rangle x_i$ and similarly for $\alpha^\vee$. We note the identities
\begin{equation}\label{e:comm1}
\sum_{i, j}x_i[y_j,x_i]y_j = (\bx\cdot\by) - \Sigma
\end{equation}
(where we used $S_{ij}=[y_i,x_j]=[y_j,x_i]=S_{ji}$ and $\alpha s_\alpha \alpha^\vee = -\alpha \alpha^\vee s_\alpha$) and
\begin{equation}\label{e:comm2}
\sum_{i, j}x_i[y_j,x_j]y_i
= n(\bx\cdot \by)+2(\bx\cdot\by)Z - 2\Sigma.
\end{equation}
That said, we compute
\begin{equation}\label{e:xysquare}
(\bx\cdot\by)^2 = \sum_{i,j} x_iy_ix_j y_j
= Q + (\bx\cdot\by) - \Sigma.
\end{equation}
Further,
using (\ref{e:comm1}), (\ref{e:comm2}) and (\ref{e:xysquare}), we get
\begin{align*}
\bM^2 &= \sum_{i<j} M_{ij}^2\\
&= \sum_{i,j} x_i^2y_j^2 - (x_ix_jy_iy_j) +x_i[y_j,x_i]y_j - x_i[y_j,x_j]y_i \\
&= \bx^2\by^2 -Q + ((\bx\cdot\by) - \Sigma) - (n(\bx\cdot\by)+2(\bx\cdot\by)Z - 2\Sigma)\\
&= \bx^2\by^2 - (\bx\cdot\by)^2 - (\bx\cdot\by)(2Z + n - 2),
\end{align*}
where we used $-Q + (\bx\cdot\by) - \Sigma = -(\bx\cdot\by)^2 + 2(\bx\cdot\by) - 2\Sigma$. This finishes the proof.
\end{proof}

\begin{Proposition}\label{p:CasM2}
The Dunkl angular momentum square and the Casimir are related via the identity
\[
\Omega_{\fs\fl(2)} = -\bM^2 + Z(Z + n - 2) + \tfrac{n(n-4)}{4} = -\bM^2 + \big(Z + \tfrac{n-2}{2}\big)^2 - 1.
\]
\end{Proposition}

\begin{proof}
We start by noting that the element $H =\tfrac{1}{2}(\bx\cdot\by + \by\cdot\bx)$ can be written as
$H = \bx\cdot\by + \tfrac{n}{2} +Z$.
Since $\bx\cdot\by$ commutes with $Z$ we have that $H^2 = (\bx\cdot\by)^2 + (2Z+n)(\bx\cdot\by) + \big(Z+\tfrac{n}{2}\big)^2$. Next, note that similarly to
(\ref{e:comm1}), using $[y_j,x_i]=[y_i,x_j]$, we have the identities
\begin{equation*}
\sum_{i,j} x_iy_j[y_j,x_i] = \sum_{i,j}x_iy_j\bigg(\delta_{ij} + \sum_{\alpha>0}c_\alpha \langle \alpha,y_i\rangle\langle x_j,\alpha^\vee \rangle s_\alpha\bigg)
= (\bx\cdot\by) + \Sigma
\end{equation*}
and
\begin{equation*}
\sum_{i,j} [y_j,x_i]y_jx_i= \sum_{i,j}\bigg(\delta_{ij} + \sum_{\alpha>0}c_\alpha \langle \alpha,y_i\rangle\langle x_j,\alpha^\vee \rangle s_\alpha\bigg)y_jx_i
= (\by\cdot\bx) + \Sigma',
\end{equation*}
where $\Sigma'= \sum_{\alpha>0}c_\alpha \alpha^\vee\alpha s_\alpha$. Similarly we have $\sum_{i,j} y_j[y_j,x_i]x_i = (\by\cdot\bx) - \Sigma'$. Recall, $X:= -\tfrac{1}{2}\bx^2$ and $Y:=\tfrac{1}{2} \by^2$. Using
$\big[y^2,x^2\big] = [y,x]yx+y[y,x]x + xy[y,x] + x[y,x]y$, we get
\begin{align*}
(-4)(XY + YX) &= \sum_{i,j} x_i^2y_j^2 + y_j^2x_i^2\\
&= 2\big(\bx^2\by^2\big) + 2(\bx\cdot\by + \by\cdot\bx)\\
&= 2\big(\bx^2\by^2\big) + 4\bx\cdot\by +2n + 4Z,
\end{align*}
from which
\begin{align*}
\Omega_{\fs\fl(2)} &= H^2 + 2(XY + YX) \\
&= (\bx\cdot\by)^2 + (2Z+n)(\bx\cdot\by) + \big(Z+\tfrac{n}{2}\big)^2 - \big(\bx^2\by^2\big) - 2(\bx\cdot\by) -n -2Z \\
&= -\bM^2 + \big(Z + \tfrac{n}{2}\big)^2 -n -2Z\\
&= -\bM^2 + \big(Z + \tfrac{n-2}{2}\big)^2 - 1,
\end{align*}
as required.
\end{proof}

\begin{Remark}\label{r:AngHamiltonian}
Comparing the computations above for $\Omega_{\fs\fl(2)}$ and the computations in \cite{FH15} for the angular Calogero--Moser Hamiltonian $H_\Omega$, we get
\[
\Omega_{\fs\fl(2)} = 2H_\Omega + \tfrac{1}{4}n(n-4).
\]
\end{Remark}

\section{Clifford algebra and AMA-Dirac elements}\label{s:Cliff}
Let $\mathcal{C}_\bbr = \mathcal{C}_\bbr(E,B)$ denote the Clifford algebra associated to the pair $(E,B)$.
The Clifford algebra $\mathcal{C}_\bbr$ is the quotient of the tensor algebra $T_\bbr(E) = \oplus_{i\geq 0} T^i(E)$ on $E$ modulo the ideal generated by
the expressions
\[y\otimes y' + y'\otimes y - 2 B(y,y')\]
for all $y,y'\in E$ (see \cite{M76} for more details). Furthermore, with respect to the canonical map $\iota\colon E\to\mathcal{C}_\bbr$, the pair $(\mathcal{C}_\bbr,\iota)$ satisfies the universal property, that, for any unital $\mathbb{R}$-algebra $A$ and any linear map $\varphi\colon E\to A$ satisfying $\varphi(y)\varphi(y') + \varphi(y')\varphi(y) = 2B(y,y')$, there is a unique algebra homomorphism $\tilde\varphi\colon \mathcal{C}_\bbr\to A$ such that $\tilde\varphi\iota = \varphi.$ For each $1\leq j \leq n$, let $c_j := \iota(y_j)$, where $\{y_1,\dots,y_n\}$ is our fixed orthonormal basis of $E$. Then, $\mathcal{C}_\bbr$ is generated by $\{c_1,\dots,c_n\}$, with Clifford relations
\begin{equation}\label{e:CliffRel}
 \{c_i,c_j\} := (c_ic_j + c_jc_i) = 2B(y_i,y_j) = 2\delta_{ij},
\end{equation}
for all $1\leq i,j\leq n$.

\subsection[Pin cover of W]{Pin cover of $\boldsymbol{W}$} \label{s:PinCover}
The reference for this part is~\cite{M76}. We have the $\mathbb{Z}_2$-grading $\mathcal{C}_\bbr = \mathcal{C}_\bbr^0\oplus\mathcal{C}_\bbr^1$, where $\mathcal{C}_\bbr^0$ is the image of $\oplus_{i\geq 0} T^{2i}(E)\subset T(E)$ while $\mathcal{C}_\bbr^1$ is the image of the odd powers in the tensor algebra. We let $\varepsilon\colon \mathcal{C}_\bbr\to\mathcal{C}_\bbr$ denote the automorphism which acts as the identity on $\mathcal{C}_\bbr^0$ and minus the identity on $\mathcal{C}_\bbr^1$. The anti-automorphism ${(\cdot)}^t$ of $T_\bbr(E)$ that sends $\eta = \eta_1\otimes\cdots\otimes\eta_p$ to $\eta^t = \eta_p\otimes\cdots\otimes\eta_1$, for all $\eta_1,\dots,\eta_p\in E$, descends to an anti-automorphism of $\mathcal{C}_\bbr$, called the {\it transpose}. Furthermore, let $\ast$ denote the anti-automorphism $\eta^* = \varepsilon(\eta^t)$, for all $\eta\in\mathcal{C}_\bbr$ and let $N(\eta) = \eta^*\eta$, for $\eta \in \mathcal{C}_\bbr$, denote the {\it spinorial norm}. Recall that the group $\Gamma = \Gamma(E,B)$ defined by
\[
\Gamma = \big\{\eta\in\mathcal{C}_\bbr^\times\mid \varepsilon(\eta)y\eta^{-1}\in E \textup{ for all } y\in E\big\}
\]
is the so-called {\it twisted Clifford group} and the homomorphism $p\colon \Gamma\to \mathsf{O}=\mathsf{O}(E,B)$, defined via $p(\eta)y = \varepsilon(\eta)y\eta^{-1}$, for all $\eta\in\Gamma$ and $y\in\fh$, is such that the sequence
\begin{equation}\label{e:CentExtSES}
1 \longrightarrow \bbr^\times \longrightarrow \Gamma \stackrel{p}{\longrightarrow} \mathsf{O} \longrightarrow 1
\end{equation}
is a short exact sequence. The {\it pinorial group} $\mathsf{Pin} = \mathsf{Pin}(E,B)$ is given by
\[
\mathsf{Pin} = \big\{\eta\in\Gamma \mid N(\eta)^2 = 1 \big\} \subset \Gamma
\]
and the sequence (\ref{e:CentExtSES}) restricts to a short exact sequence
\begin{equation}\label{e:WSES}
1 \longrightarrow \{\pm 1\} \longrightarrow \mathsf{Pin} \stackrel{p}{\longrightarrow} \mathsf{O} \longrightarrow 1.
\end{equation}
The Pin-cover of $W\subset\mathsf{O}$ is defined as $\tilde W := p^{-1}(W)\subset \mathsf{O}$. Given a coroot $\alpha^\vee\in R^\vee\subset E$, recall that we can write $\alpha^\vee = \sum_i\lpi x_i,\alpha^\vee \rpi y_i$. Using (\ref{e:pairings}), note that
\[
\tfrac{1}{|\alpha^\vee|}\iota(\alpha^\vee)
= \tfrac{1}{|\alpha^\vee|}\sum_i B(y_i,\alpha^\vee) c_i
= \tfrac{1}{|\alpha|}\sum_i B^*(x_i,\alpha) c_i
= \tfrac{1}{|\alpha|}\iota\big(B^{-1}(\alpha)\big).
\]
We are thus justified to abuse the notation and define, for any $\alpha\in R$, \[\tilde{s}_\alpha := |\alpha^\vee|^{-1}\alpha^\vee \in \mathcal{C}_\bbr.\]
One can show (see \cite[Proposition~2.6]{M76}) that $p(\tilde{s}_\alpha) = s_\alpha$. Further, $p^{-1}(s_\alpha) = \{\pm \tilde{s}_\alpha\}$. Then, with respect to generators and relations, we have that (see \cite[Theorem~4.2]{M76}), on the one hand~$W$ has presentations
\begin{gather*}
W = \langle s_\alpha,\alpha\in R\mid s_\alpha^2=1,s_\alpha s_\beta s_\alpha = s_\gamma,\gamma = s_\alpha(\beta)\rangle,
\\
W = \langle s_\alpha,\alpha\in \Delta \mid (s_\alpha s_\beta)^{m_{\alpha,\beta}} = 1 \rangle
\end{gather*}
while the double-cover has presentations
\begin{gather}\label{e:PinPresentation}
\tilde{W} = \langle \theta, \tilde{s}_\alpha,\alpha\in R\mid \tilde{s}_\alpha^2=1=\theta^2,\tilde{s}_\alpha \tilde{s}_\beta \tilde{s}_\alpha = \sigma\tilde{s}_\gamma,\gamma = s_\alpha(\beta), \theta \text{ central}\rangle,\\
\label{e:PinPresentationmalpha}
\tilde{W} = \langle \theta, \tilde{s}_\alpha,\alpha\in \Delta \mid (\tilde{s}_\alpha \tilde{s}_\beta)^{m_{\alpha,\beta}} = (\theta)^{m_{\alpha,\beta}-1}, \theta \text{ central}\rangle.
\end{gather}

 We let $\mathcal{C} = \mathcal{C}_\bbr\otimes \bbc$ be the complexification. Letting $\theta = -1 \in \Cc$ the group $\tilde{W}$ is a subgroup of $\Pin \subset\Cc$. However, the group algebra $\mathbb{C}\tilde{W}$ does not inject into $\Cc$. Decomposing the identity as two idempotents $1 = \tfrac{1}{2}(1+\theta) + \tfrac{1}{2}(1-\theta)$, the group algebra $\mathbb{C}\tilde{W}$ splits as a direct sum of two algebras
 \begin{equation}\label{eq::twistedgroupalg} \mathbb{C}\tilde{W} =\mathbb{C}\tilde{W}_+ \oplus \mathbb{C}\tilde{W}_-,
 \end{equation}
 where the central element $\theta$ is specialised to either $+1$ or $-1$ in $\mathbb{C}\tilde{W}_+$ and $\mathbb{C}\tilde{W}_-$ respectively. The algebra $\mathbb{C}\tilde{W}_+$ is isomorphic to $\mathbb{C}W$. Following \cite{K05}, we refer to the algebra $\mathbb{C}\tilde{W}_-$ as the twisted group algebra.

Note that $\mathcal{C}$ has the same presentation by generators and relations as in (\ref{e:CliffRel}). As is well-known, if $n=\dim_\bbr(E)$, then $\mathcal{C}$ has one (resp.\ two) equivalence classes of complex irreducible representations of dimension $2^{\lfloor n/2 \rfloor}$ for $n$ even (resp.~$n$ odd). Let also $\ast$ denote the anti-linear extension to $\Cc$ of the anti-involution $\eta^\ast = \varepsilon(\eta^t)$ defined above.
 Finally, we let $\rho\colon \bbc\tilde W\to \ama\otimes\mathcal{C}$ denote the homomorphism obtained from the diagonal embedding of~$\tilde W$ defined by
\begin{equation}\label{ed:rhodef}
\rho(\tilde w) = p(\tilde w)\otimes \tilde w
\end{equation}
 for all $\tilde w\in\tilde W$ and extended linearly, where $p\colon \tilde W\to W$ is the double-cover projection map and~$\tilde{w}$ is considered as an element in $\Pin \subset \Cc$.

\subsection{AMA-Dirac elements}
Both algebras $\mathsf{H}$ and $\mathcal{C}$
contain a copy of the vector space $\wedge^2\fh$ with basis
$\{\uM_{ij}\mid 1\leq i<j \leq n\}$. In $\mathsf{H}$, these are realised by the elements $M_{ij} = x_iy_j-x_jy_i$ for $1\leq i<j \leq n$ that forms part of the generating set of $\mathsf{A}$ and in $\mathcal{C}$ they are realised by quadratic elements $c_ic_j\in\mathcal{C}$. In what follows, we may use the short hand notation $Y$ to denote $Y \otimes 1 \in \ama \otimes \Cc$ for any $Y \in \ama$. For example, $Y$ may be $M_{ij}$ or $w \in W$.

\begin{Definition}
The {\it Dirac element of the angular momentum algebra} is defined by
\begin{equation*}
 \mathcal{D} = \sum_{i<j} M_{ij}\otimes c_ic_j \in \mathsf{A}\otimes \mathcal{C}.
\end{equation*}
For brevity, we shall refer to this element as the AMA-Dirac element.
\end{Definition}

\begin{Proposition}
The AMA-Dirac element is independent of the choice of orthonormal basis $\{y_1,\dots,y_n\}$ made. In particular, it is $\rho\big(\tilde W\big)$-invariant.
\end{Proposition}

\begin{proof}
The $\rho\big(\tilde W\big)$-invariance follows from the independence of the basis since conjugating $\mathcal{D}$ by $\rho(\tilde{s}_\alpha) = s_\alpha\otimes\tilde{s}_\alpha$ causes us to write the expression for~$\mathcal{D}$ with respect to the bases $\{s_\alpha(y_1),\dots,s_\alpha(y_n)\}$ and $\{s_\alpha(x_1),\dots,s_\alpha(x_n)\}$.

The proof for the independence of the choice of basis is standard, and we briefly recall the steps. If $\{y'_1,\dots,y'_n\}$ is another choice, we have $y'_j = \sum_k Q_{jk}y_k$ and $x'_j = B(y'_j) = \sum_k Q_{jk}x_k$ where the collection $\{Q_{jk}\mid 1\leq j,k \leq n\}$ satisfy $\sum_k Q_{ik}Q_{jk} = \delta_{ij}$. It is then straightforward to check that
\[
2\mathcal{D}' = \sum_{i,j}M'_{ij}\otimes c'_ic'_j = \sum_{k,l}M_{kl}\otimes c_kc_l = 2\mathcal{D},
\]
where $M'_{ij} = x'_iy'_j - x'_jy'_i \in \ama$ and $c'_i=\iota(y'_i)\in\Cc$.
\end{proof}

As in every Dirac theory, we now compute the square of the AMA-Dirac element. We will show that upon subtracting a correction term this element yields a square-root of the Casimir~$\Omega_{\mathfrak{sl}(2)}$, modulo a constant. Before we compute $\mathcal{D}^2$, we shall need some preliminary computations.

Let $\Pi = \big\{(i,j)\in \mathbb{Z}^2;\, 1\leq i<j\leq n\big\}$. Note that we can write the Cartesian product as the disjoint union
\begin{equation}\label{e:CartDecomp}
\Pi^2 = \Pi_0 \cup \Pi_1 \cup \Pi_2
\end{equation}
where $\Pi_q := \big\{((i,j),(k,l))\in\Pi^2;\, |\{i,j\}\cap\{k,l\}|=q \big\}$, for $q \in \{0,1,2\}$. If $\pi = (i,j) \in \Pi$, we shall write $c_\pi = c_ic_j$ in the Clifford algebra and $M_\pi = M_{ij}$ in $\mathsf{A}$. Then,
\begin{equation}\label{e:squaredecomp}
\mathcal{D}^2 = \sum_{(\pi,\sigma) \in \Pi^2 } M_\pi M_\sigma \otimes c_\pi c_\sigma = \Sigma_0 + \Sigma_1 + \Sigma_2,
\end{equation}
where $\Sigma_q$ is the sum over $\Pi_q$, in the decomposition (\ref{e:CartDecomp}).

\begin{Lemma}
With notations as in \eqref{e:squaredecomp}, we have $\Sigma_2 = -\mathbf{M}^2$ and $\Sigma_0 = 0$.
\end{Lemma}

\begin{proof}
As $(c_ic_j)^2 = -1$ when $i\neq j$, it immediately follows that $\Sigma_2 = -\mathbf{M}^2$. As for $\Sigma_0$, to each pair $((i,j),(k,l)) \in \Pi_0$, noting that $[c_ic_j,c_kc_l]=0$, after ordering the $4$-tuple $i<j<k<l$,
and fixing the Clifford element $c_ic_jc_kc_l$ to the right-hand side of the tensor product, the contribution on the left-hand side becomes
\begin{equation*}
(M_{ij}M_{kl} + M_{kl}M_{ij} - M_{ik}M_{jl} - M_{jl}M_{ik} + M_{il}M_{jk} + M_{jk}M_{il})\otimes c_ic_jc_kc_l,
\end{equation*}
from which we obtain
\begin{gather*}
 \Sigma_0 = \sum_{1\leq i<j<k<l \leq n}
 2(M_{ij}M_{kl} + M_{jk}M_{il} + M_{ki}M_{jl})\otimes c_ic_jc_kc_l\\
\hphantom{\Sigma_0 = \sum_{1\leq i<j<k<l \leq n}}{}
+ ([M_{kl},M_{ij}] + [M_{il},M_{jk}] + [M_{jl},M_{ki}])\otimes c_ic_jc_kc_l.
\end{gather*}
Using the relation (\ref{e:agacomrel}) of $\mathsf{A}$ and the symmetry $S_{ab} = S_{ba}$ for any indices $a$, $b$, we obtain
\begin{equation*}
 [M_{kl},M_{ij}] + [M_{il},M_{jk}] + [M_{jl},M_{ki}] =
 -2(M_{ij}S_{kl} + M_{jk}S_{il} + M_{ki}S_{jl}),
\end{equation*}
from which, using now (\ref{e:agacrossrel}), we obtain $\Sigma_0=0$.
\end{proof}

\begin{Lemma}With notations as in \eqref{e:squaredecomp}, we have $\Sigma_1 = (n-2)\mathcal{D} + \{\mathcal{D},Z\}$.
\end{Lemma}

\begin{proof}
Each pair $((i,j),(k,l)) \in \Pi_1$ has exactly three distinct entries. Using the Clifford relations, each product $c_\pi c_\sigma$ with $(\pi,\sigma)\in\Pi_1$ reduces to a product of the type $c_ic_j$, for distinct indices $i$, $j$. For example, $c_ic_kc_jc_k = -c_ic_j$ and so on.
Moreover, we can label the sum $\Sigma_1$ in terms of ordered triples $(i<j<k)$ and we obtain
\begin{equation*}
 \Sigma_1 = \sum_{i<j<k} [M_{ik},M_{ij}]\otimes c_jc_k +
 [M_{ij},M_{jk}]\otimes c_ic_k + [M_{jk},M_{ik}]\otimes c_ic_j,
\end{equation*}
which, after applying the relations of $\mathsf{A}$ and the symmetry $S_{ab} = S_{ba}$ for the indices, yields
\begin{gather*}
\Sigma_1 = \sum_{i<j<k}\big\{(M_{jk}S_{ii} - M_{ji}S_{ik} - M_{ik}S_{ji})\otimes c_jc_k
+ (M_{ik}S_{jj} - M_{ij}S_{jk} - M_{jk}S_{ij})\otimes c_ic_k \\
\hphantom{\Sigma_1 = \sum_{i<j<k}}{}
 + (M_{ij}S_{kk} - M_{ik}S_{kj} - M_{kj}S_{ik})\otimes c_ic_j\big\}.
\end{gather*}
Thus, each Clifford element $c_ic_j$, contributes to the sum $\Sigma_1$ with the quantity $C(i,j)\in\mathsf{A}$ given by
\begin{align*}
C(i,j) &= \sum_{k\notin\{i,j\}} (M_{ij}S_{kk} - M_{ik}S_{kj} - M_{kj}S_{ik})\\
&= M_{ij}(n+2Z) - \sum_{k=1}^n (M_{ik}S_{kj} + M_{kj}S_{ik}).
\end{align*}
Furthermore, denoting $\epsilon(i,j) = \sum_{k=1}^n (M_{ik}S_{kj} + M_{kj}S_{ik})$, we obtain
\begin{align*}
 \epsilon(i,j) &= 2M_{ij} +
 \sum_{\alpha>0}c_\alpha (\alpha(\langle x_i,\alpha^\vee\rangle y_j - \langle x_j,\alpha^\vee\rangle y_i) - (\langle\alpha,y_i\rangle x_j - \langle\alpha,y_j\rangle x_i)\alpha^\vee)s_\alpha\\
 &= 2M_{ij} + \sum_{\alpha>0}c_\alpha(M_{ij}s_\alpha - s_\alpha M_{ij}).
\end{align*}
We conclude, therefore,
that
\begin{align*}
 \Sigma_1 &= \sum_{i<j}C(i,j)\otimes c_ic_j\\
 &=(n-2)\mathcal{D} + 2\mathcal{D}Z + [Z,\mathcal{D}],
\end{align*}
and the claim follows from $\{Z,\mathcal{D}\} = 2\mathcal{D}Z + [Z,\mathcal{D}]$.
\end{proof}

\begin{Theorem}
We have
\[
\mathcal{D}^2 = -\mathbf{M}^2 + (n-2)\mathcal{D} + \{\mathcal{D},Z\}.
\]
\end{Theorem}

\begin{proof}
Follows directly from the previous lemmas and the identity (\ref{e:squaredecomp}).
\end{proof}

\begin{Corollary}\label{c:DiracSquare}
Let $\phi := \tfrac{1}{2}(2Z + n-2)$. The element $\mathfrak{D}_0 = (\mathcal{D} - \phi)$ is a square root of a Casimir element of $\mathfrak{sl}(2)$.
\end{Corollary}

\begin{proof}
We compute directly to get
\begin{align*}
 \mathfrak{D}_0^2 &= \mathcal{D}^2 + \phi^2 -\{\mathcal{D},\phi\}\\
 &= -\mathbf{M}^2 + (n-2)\mathcal{D} + \{\mathcal{D},Z\} - \big\{\mathcal{D},Z + \tfrac{n-2}{2}\big\} + \big(Z + \tfrac{n-2}{2}\big)^2\\
 &= \Omega_{\mathfrak{sl}(2)} + 1,
\end{align*}
as required, where use was made of Proposition \ref{p:CasM2} in the last equality.
\end{proof}

\section[AMA-Dirac and the SCasimir of osp(1|2)]{AMA-Dirac and the SCasimir of $\boldsymbol{\mathfrak{osp}(1|2)}$}\label{s:SCasimir}
Now recall (see for example \cite{DOV18a} and \cite{DOV18b}) that the algebra $\mathsf{H}\otimes\mathcal{C}$ contains a copy of the Lie superalgebra $\mathfrak{osp}(1|2)$ spanned by the Lie triple $(H,X,Y)\subset \mathsf{H}$ together with the elements
\[
\uD = \sum_i y_i\otimes c_i,\qquad \ux =\sum_i x_i\otimes c_i,
\]
of $\mathsf{H}\otimes\mathcal{C}$. The element $\uD$ is often referred to as the Dunkl--Dirac operator, as it squares to the Dunkl--Laplace operator when viewed as an operator on the polynomial space. Next, we relate the AMA-Dirac element with the SCasimir $\mathcal{S}$ of $\mathfrak{osp}(1|2)$.

\begin{Proposition}
We have the following identity:
\[-2\mathcal{D} = [\uD,\ux] - (n+2Z) = [\uD,\ux] - 2(\phi+1).\]
\end{Proposition}

\begin{proof}
Using that, for all $i\neq j$, we have $y_ix_j - y_jx_i = x_jy_i +S_{ij} - x_iy_j - S_{ji} = x_jy_i - x_iy_j = -M_{ij}$ in $\mathsf{H}$, it is straightforward to compute
\begin{align*}
 [\uD,\ux] &= \sum_{i,j}(y_ix_j\otimes c_ic_j - x_jy_i\otimes c_jc_i)\\
 &=\sum_{i,j}( (y_ix_j + x_jy_i)\otimes c_ic_j -x_jy_i\otimes 2\delta_{ij}\\
 &=\sum_i[y_i,x_i]\otimes 1 + \sum_{i< j}(y_ix_j + x_jy_i - y_jx_i - x_iy_j)\otimes c_ic_j\\
 &=(n+2Z)\otimes 1 -2\mathcal{D},
\end{align*}
where, in the last equation, we used~(\ref{e:commutatorsum}). The claim now follows immediately.
\end{proof}

\begin{Corollary}
As elements of $\mathsf{H}\otimes\mathcal{C}$, the AMA-Dirac element and the SCasimir of~$\mathfrak{osp}(1|2)$ satisfy
$\mathcal{D} +\mathcal{S}=\frac{1}{2} + \phi$.
\end{Corollary}

\begin{proof}
With our notational conventions, the SCasimir of $\mathfrak{osp}(1|2)$ is given by
$\mathcal{S} = \tfrac{1}{2}([\uD,\ux] - 1)$ (see \cite[equation~(3.3)]{DOV18a}). The claim follows from the previous proposition.
\end{proof}

\section{Vogan's conjecture and Dirac cohomology}\label{s:Cohomology}

Inspired by \cite{BCT12} we prove an analogue of Vogan's conjecture in the context of the angular momentum algebra~$\mathsf{A}$. Vogan's original conjecture states that if the Dirac cohomology for a $(\mathfrak{g},K)$ module $X$ is non-zero then the infinitesimal character of~$X$ can be described in terms of the highest weight of a $\tilde{K}$-type in the cohomology. This conjecture was proved in~\cite{HP02}. However, in our context, instead of a single Dirac operator relating the center of the algebra in question and the centre of (the double-cover of) the Weyl group, we shall construct a family of operators depending on central elements.

\subsection{An analogue of Vogan's conjecture}
Denote by $Z\tilde W$ the centre of $\bbc \tilde W$. Denote also by $\ast$ the anti-linear involution of $\Cc$ defined in Section~\ref{s:PinCover} restricted to $\tilde W$ and extended anti-linearly to a star operation on $\bbc \tilde{W}$.

\begin{Definition}\label{d:hatW} Let $\tilde{W}$ be the double cover of $W$. If $w_0 \neq(-1)_\fh$, we define $\hat{W} = \tilde{W}$. Else, we define $\hat{W} = \tilde{W} \times C_2$, where, abusing notation, $(-1)_\fh$ generates $C_2$. In this case, we extend the homomophism $\rho$ of (\ref{ed:rhodef}) to $\hat{W} \to \ama \otimes \Cc$ by $\rho((-1)_\fh) = (-1)_\fh \otimes 1$. Furthermore, extend $*$ to $\hat{W}$ by $(-1)_\fh^*=(-1)_\fh^{-1} = (-1)_\fh$.
\end{Definition}

When $w_0 = (-1)_\fh$, the algebra $\bbc \hat{W}$ is a central extension of $\bbc \tilde{W}$. In this case, there is a~$2$-to-$1$ map from $Z\hat{W}$ to $Z\tilde{W}$.

\begin{Definition}\label{d:Dirac_C}
An element $C\in \bbc\hat W$ is called {\it admissible} if $C$ is central and $C^* = C$. For any admissible $C\in Z\hat{W}$, define
\begin{equation}\label{e:CDirac}
\mathfrak{D}_C := (\mathcal{D} - \phi) + \rho(C),
\end{equation}
where $\mathcal{D}$ is the AMA-Dirac element and $\phi = \tfrac{1}{2}(2Z + n -2)$.
\end{Definition}

\begin{Remark}
The set of admissible elements has the structure of a real vector space,
and it is not empty as $C=0$ is admissible.
In the next section we shall exhibit and study a more interesting admissible element.
\end{Remark}

\begin{Theorem}\label{t:Vogan}
Given an admissible $C \in \bbc\hat W$, there is an algebra homomorphism
\[
\zeta_C\colon \ Z(\ama) \to Z\hat W
\]
such that, for all $z \in Z(\ama)$ there exists $a\in \ama\otimes \mathcal{C}$ satisfying
\[
z\otimes 1 = \rho(\zeta_C(z)) + \mathfrak{D}_Ca + a\mathfrak{D}_C.
\]
\end{Theorem}

\begin{proof}
In this proof, we abbreviate $\Omega = \Omega_{\fs\fl(2)}$.
Because $Z(\ama) = \mathcal{R}[\Omega]$ has a very simple algebraic structure, we can give a straightforward proof, without having to use the more sophisticated ideas from \cite{HP02}. Let $\gamma := \rho\big(C^2\big) - 1\in \rho\big(Z\hat{W}\big)$. We show, by induction, that for every $m\in\mathbb{Z}_{\geq 1}$, there is $a_m \in \ama \otimes \Cc$ such that
\[
\Omega^m\otimes 1 = \gamma^m + \{\mathfrak{D}_C,a_m\}.
\]
Indeed, since $C\in \bbc\tilde W$ we have that $\rho(C)$ commutes with $\mathfrak{D}_0$ and $\mathfrak{D}_C$. Thus
\begin{align*}
 \mathfrak{D}_C^2 &= \mathfrak{D}_0^2 + \rho(C)^2 + 2\mathfrak{D}_0\rho(C)\\
 &= \Omega + 1 + \rho(C)(\rho(C) + 2\mathfrak{D}_0)\\
 &= \Omega + 1 - \rho(C)^2 + 2\rho(C)\mathfrak{D}_C,
\end{align*}
from which we conclude that upon defining $a_1 := \tfrac{1}{2}\mathfrak{D}_C-\rho(C)$, we have
\[
\Omega\otimes 1 = \Omega = \gamma + \{\mathfrak{D}_C,a_1\}.
\]
Note that $a_1$ commutes with $\mathfrak{D}_C$.
Now assume we have $\Omega^m = \gamma^m + \{\mathfrak{D}_C,a_m\}$ for some $a_m\in \ama \otimes \Cc$ that commutes with $\mathfrak{D}_C$. It is then straightforward to compute that
\begin{align*}
 \Omega^{m+1} &= (\gamma^m + \{\mathfrak{D}_C,a_m\})(\gamma + \{\mathfrak{D}_C,a_1\})\\
 &=\gamma^{m+1} + \{\mathfrak{D}_C,a_{m+1}\},
\end{align*}
with $a_{m+1} := a_m\gamma + a_1\gamma^m + 2\mathfrak{D}_Ca_ma_1$. Therefore, the homomorphism $\zeta_C$ is defined by $\zeta_C(\Omega) = \big(C^2 -1\big)$ and $\zeta_C((-1)_\fh) = (-1)_\fh$, if $w_0 = (-1)_\fh$.
\end{proof}

\begin{Remark}
In the proof of the previous theorem we only used that $C$ was in $\bbc\hat{W}$. The conditions on admissibility are needed below to ensure that the operators we obtain are self-adjoint.
\end{Remark}

\subsection{Unitary structures}
Let $\ast$ denote the anti-linear anti-involution of $\Cc$ defined in Section~\ref{s:PinCover}. Let also $\bullet$ be the restriction to $\ama$ of the anti-linear anti-involution of $\rca$ characterized on the generators by $x_i^\bullet = y_i$, $y_i^\bullet = x_i$ and $w^\bullet = w^{-1}$, for all $1\leq i \leq n$ and $w\in W$, where we recall that we have fixed orthonormal bases of $E$ and $E^*$. We then define an anti-linear anti-involution $\star$ on $\ama\otimes\Cc$ by taking the tensor product of these two anti-involutions. It is straightforward to check that $\rho(\tilde{w})^\star = \rho(\tilde{w}^*)$ for any $\tilde{w}\in\hat{W}$, where $\rho\colon \bbc\hat{W}\to\ama\otimes\Cc$ is the homomorphism of Definition~\ref{d:hatW}.

Now fix, once and for all, $(\sigma,S)$ an irreducible module for $\Cc$.
Endow $S$ with a unitary structure $(-,-)_S$, i.e., a complex inner product on $S$ that is also $\ast$-Hermitian
\[
(\sigma(\eta)s_1,s_2)_S = (s_1,\sigma(\eta^*)s_2)_S,
\]
for all $\eta\in\Cc$ and $s_1,s_2\in S$. For any $\bullet$-Hermitian module $(\pi,X)$ of $\ama$ we endow $X\otimes S$ with a~$\star$-Hermitian structure $(x\otimes s,x'\otimes s')_{X\otimes S} = (x,x')_X(s,s')_S$ for all $x,x'\in X$ and $s,s' \in S$. If the $\star$-Hermitian form on $X \otimes S$ is also positive definite, then we say $X \otimes S$ is unitary. We define operators in $\End(X\otimes S)$ by taking the image of the AMA-Dirac elements under $\pi\otimes\sigma$.

\begin{Proposition}
If $(\pi,X)$ is a $\bullet$-Hermitian $\ama$-module, then the operators $D = (\pi\otimes\sigma)(\mathcal{D})$ and $D_C = (\pi\otimes\sigma)(\mathfrak{D}_C)$, for admissible $C\in Z\hat{W}$, are self-adjoint. Furthermore, if $X\otimes S$ is unitary, then
\[
\big(D_C^2(x\otimes s),x\otimes s\big)_{X\otimes S} \geq 0
\]
for all $x \in X$ and all $s \in S$.
\end{Proposition}

\begin{proof}
It is straightforward to check that $M_{ij}^\bullet = -M_{ij}$ and $(c_ic_j)^* = -(c_ic_j)$, from which we get that the AMA-Dirac element is invariant for the $\star$-involution and thus $D$ is indeed Hermitian. Also, it is straightforward to check that $\phi^\bullet = \phi$ and the claims follow since $C$ is admissible.
\end{proof}

\begin{Example}\label{e:harmonicex}
Fix $\tau$ an irreducible representation of $W$ and let $M_c(\tau)$ be the standard module at $\tau$ for the rational Cherednik algebra $\rca$ with $\dim(E) \geq 2$. For real parameter functions $c$ close enough to $c=0$, it is known (see \cite{ES09}) that $M_c(\tau)$ is a unitary $\rca$-module. For such parameters, the modules $X_c(\tau)_m = \ker(\Delta_c)\cap M_c(\tau)_m$ are irreducible unitary $\ama$-modules (see \cite[Theorem~B]{CDM20}), where $\Delta_c$ is the Dunkl--Laplacian and $M_c(\tau)_m$ are the homogeneous elements of degree $m$ of $M_c(\tau)$. Let $\lambda(c,\tau,m)=m + \tfrac{n}{2} + N_c(\tau)$, where $N_c(\tau)$ is the scalar on which the central element $Z = \sum_{\alpha>0} c_\alpha s_\alpha$ acts on $\tau$. Then, the Casimir $\Omega = \Omega_{\fs\fl(2)}$ acts on $X_c(\tau)_m$by the scalar $\chi = \lambda(c,\tau,m)(\lambda(c,\tau,m) - 2)$. From the previous proposition, with $C=0$, we get that the parameter function $c$ for unitary $M_c(\tau)$ satisfy $\chi \geq -1$.
\end{Example}

\subsection{Dirac cohomology}

In the proof of Theorem \ref{t:Vogan}, we computed the square
\begin{equation}\label{e:C-Diracsquare}
 \mathfrak{D}_C^2 = \Omega_{\mathfrak{sl}(2)} - \big(\rho(C)^2 - 1\big) + 2\rho(C)\mathfrak{D}_C,
\end{equation}
for any admissible $C\in Z\hat W$. Thus, in the kernel of a Dirac operator $D_C = (\pi\otimes\sigma)(\mathfrak{D}_C) \in \End(X\otimes S)$, where $(\pi,X)$ is an $\ama$-module, we get the equation
\[
(\pi\otimes\sigma)(\Omega_{\fs\fl(2)}) = (\pi\otimes\sigma)\big(\rho(C)^2 -1\big).
\]

Suppose $w_0=(-1)_\fh$. Then $w_0$ is central in $\ama$. Hence $(-1)_\fh$ acts by a scalar on $X$ and $X \otimes S$. The element $(-1)_\fh$ squares to $1$ and therefore this scalar is $1$ or $-1$. Because $(-1)_\fh$ commutes with $\mathfrak{D}_C$, it acts on the kernel of a Dirac operator $D_C$ by the same scalar. We can relate the action of the whole centre $Z(\ama)$ with the isotypic component of irreducible $\bbc\big[\hat{W}\big]$-representations occuring in the kernel of $D_C$.

\begin{Definition}\label{d:DiracCoh}
Let $(\pi,X)$ be an $\ama$-module and $C$ be an admissible element in $Z\tilde W$. The {\it Dirac cohomology of $C$} is defined by
\[
H(X,C) = \frac{\ker(D_C)}{\ker(D_C)\cap \im(D_C)},
\]
where $D_C = (\pi\otimes\sigma)(\mathfrak{D}_C) \in \End(X\otimes S)$.
\end{Definition}

\begin{Proposition}\label{p:unitarycoh}
The Dirac cohomology of $C$ is a $\hat W$-module. Moreover, if $X$ is a $\bullet$-Hermitian $\ama$-module, then $H(X,C) = \ker(D_C)$.
\end{Proposition}

\begin{proof}
Clear, as $\mathfrak{D}_C$ is $\rho(\hat{W})$-invariant and $D_C$ is self-adjoint when $X$ is $\bullet$-Hermitian.
\end{proof}

We finish this section by showing that the Dirac cohomology of $C$ {\it determines the central character} of an $\ama$-module. To make this statement precise, we need some definitions. First, we say that an $\ama$-module $(\pi,X)$ {\it has central character $\chi\colon Z(\ama)\to \bbc$} if the centre $z \in Z(\ama)$ acts by the scalar~$\chi(z)$ on~$X$.

\begin{Remark}
Every irreducible $\ama$-module has a central character. However, to the best of the authors' knowledge, the representation theory of $\ama$ is currently unknown and since $\ama$ is the deformation of the image of the universal enveloping algebra of the Lie algebra $\fs\fo(n)$ into a~smash-product of $W$ and a Weyl algebra, there might be non-irreducible $\ama$-modules with central character resembling Verma modules.
\end{Remark}

\begin{Definition}
Let $C\in Z\hat{W}$ be an admissible element and $\zeta_C\colon Z(\ama)\to Z\hat{W}$ be the homomorphism of Theorem~\ref{t:Vogan}. For any irreducible $\hat{W}$ representation $\hat{\tau}$, define the homomorphism $\chi_{\hat{\tau}}\colon Z(\ama)\to\bbc$ via
\[
\chi_{\hat{\tau}}(z) = \frac{1}{\dim\hat{\tau}}\Tr\big( \hat{\tau}(\zeta_C(z)) \big),
\]
for any $z$ in $Z(\ama)$.
\end{Definition}

\begin{Theorem}\label{t:CentChar}
Let $C\in Z\hat{W}$ be an admissible element, $\hat\tau$ be an irreducible $\hat{W}$ representation and $(\pi,X)$ be an $\ama$-module with central character $\chi$. Suppose that
\[
\Hom_{\hat{W}}\big(\hat{\tau},H(X,C)\big)\neq 0.
\]
Then, $\chi = \chi_{\hat{\tau}}$.
\end{Theorem}

\begin{proof}
The proof is mutatis mutandis of the one in \cite[Theorem~4.5]{BCT12}, but we add the short proof here, for convenience. The assumption in the statement implies the existence of a non-zero element $\xi$ in the $\hat{\tau}$-isotypic component of $X\otimes S$ which is in~$\ker(D_C)$ but not in~$\im(D_C)$. For any $z\in Z(\ama)$, since both $z\otimes 1$ and $\rho(\zeta_C(z))$ act by a scalar on $\xi$, we get, using Theorem~\ref{t:Vogan} that
\begin{align*}
 (\pi\otimes \sigma)(z\otimes 1 - \rho(\zeta_C(z)))\xi & = (\pi\otimes \sigma)(\mathfrak{D}_Ca + a\mathfrak{D}_C)\xi\\
 &= (\pi\otimes \sigma)(\mathfrak{D}_Ca)\xi\\
 &= 0,
\end{align*}
since otherwise $\xi$ would be in the image of $D_C$, which it is not. The claim follows.
\end{proof}

\begin{Remark}
In the proof of Theorem \ref{t:Vogan} we actually proved that the element ``$a$'' commuted with $D_C$, so $D_Ca + aD_C = 2aD_C$. Thus, the last bit of the proof of the previous theorem can be simplified, in our context.
\end{Remark}

\section{Examples of non-trivial admissible elements}\label{s:AdmElem}

In this last section we explore the set of admissible elements. Throughout this section, we assume the parameter function $c$ is real valued. Let
\begin{equation}\label{e:evenadm}
C_2 := \tfrac{1}{4}\sum_{\alpha,\beta>0}c_\alpha c_\beta \tilde{s}_\alpha\tilde{s}_\beta \in \bbc\tilde W\subset \bbc\hat{W}.
\end{equation}

\begin{Proposition}
The element $C_2$ of \eqref{e:evenadm} is admissible.
\end{Proposition}

\begin{proof}
Note that the element
$\tilde{Z} = \frac{1}{2} \sum_{\alpha >0} c_\alpha \tilde{s}_\alpha$ is such that $C_2 = \tilde{Z}^2$. Since $\tilde{s}_\alpha^* = \theta\tilde{s}_\alpha$ for any $\alpha\in R$, we get $\tilde{Z}^*=\theta\tilde{Z}$ and hence $C_2^* = \big(\tilde{Z}^2\big)^* = C_2$. Recall $\theta$ is the central involution in $\tilde{W}$ (see presentation~(\ref{e:PinPresentation})).
Equation (\ref{eq::twistedgroupalg}) shows that $\mathbb{C}\tilde{W}$ splits into $\mathbb{C}W \oplus \mathbb{C}\tilde{W}_-$ and $\tilde{Z}$ decomposes as
\[\tilde{Z} = \tfrac{1}{2}(1+\theta) \tilde{Z} + \tfrac{1}{2}(1-\theta) \tilde{Z},\]
where the element $ \frac{1}{2}(1
+\theta) \tilde{Z}$ is equal to the central element $Z = \frac{1}{2}\sum_{\alpha >0}c_\alpha s_\alpha$ of $ \mathbb{C}W$ and we denote $T:=\tfrac{1}{2}(1-\theta) \tilde{Z}$. Hence, we are left to prove that the element $T^2= \frac{1}{2}(1-\theta) \tilde{Z}^2$ is central in $\mathbb{C}\tilde{W}_-$. Let $\tau_\alpha = \frac{1}{2}(1-\theta)\tilde{s}_\alpha$ for any $\alpha \in R$. Presentation (\ref{e:PinPresentationmalpha}) shows that the simple `pseudo' reflections $\tau_\alpha, \alpha \in \Delta$ generate $\mathbb{C}\tilde{W}_-$. Hence, it is sufficient to prove that $T^2$ commutes with every simple pseudo reflection. To that end, note that we can express $T$ in terms of the pseudo reflections as
\[\tfrac{1}{2}(1-\theta) \tilde{Z} = T = \tfrac{1}{2}\sum_{\alpha >0} c_\alpha \tau_\alpha.
\]
Further, for $\beta \in \Delta$, write $R_+ = R_\beta \cup \Gamma_\beta$ where $R_\beta = R_+\cap s_\beta(R_-) = \{\beta\}$ and $\Gamma_\beta = R_+\cap s_\beta(R_+) = R_+ \setminus \{\beta\}$ is the complement. Write also

\begin{equation}\label{e:tildeZsums}
\tilde{\Gamma}_\beta = \tfrac{1}{2} \sum_{\alpha \in \Gamma_\beta} c_\alpha \tau_\alpha
\end{equation}
so that $T = \frac{1}{2}c_\beta \tau_\beta + \tilde{\Gamma}_\beta \in \mathbb{C} \tilde{W}_-$.
Since $\beta$ is simple then $s_\beta$ permutes the elements of $\Gamma_\beta$, so, in view of the presentation (\ref{e:PinPresentation}), and using the invariance of the parameter function~$c$, we conclude that, the anti-commutator $\big\{\tau_\beta,\tilde{\Gamma}_\beta\big\}= 0$ in $\mathbb{C}\tilde{W}_-$ and hence
\[
T^2 = \tfrac{1}{4}{c_\beta}^2\tau_\beta^2 + \tilde{\Gamma}_\beta^2 +\big\{\tfrac{1}{2}c_\beta\tau_\beta,\tilde{\Gamma}_\beta\big\}
= \tfrac{1}{4}{c_\beta}^2\tau_\beta^2 + \tilde{\Gamma}_\beta^2 - \big\{\tfrac{1}{2}c_\beta\tau_\beta,\tilde{\Gamma}_\beta\big\}
= s_\beta(T)^2.
\]
It then follows that $\tau_\beta \tilde{Z}^2 = s_\beta\big(\tilde{Z}\big)^2\tau_\beta = \tilde{Z}^2\tau_\beta$ in $\mathbb{C}\tilde{W}_-$, and we are done.
\end{proof}

\begin{Definition}
Define elements
\[T_i = \tfrac{1}{2}\sum_{\alpha \in R_+} c_\alpha \frac{\langle x_i,\alpha^\vee\rangle}{|\alpha^\vee|} s_\alpha \qquad \text{and}\qquad T_i^\bullet =\tfrac{1}{2} \sum_{\alpha \in R_+} c_\alpha \frac{\langle \alpha,y_i\rangle}{|\alpha|} s_\alpha \in \mathbb{C}W.\] \end{Definition}
Using equation (\ref{e:pairings}) then $\frac{\langle x_i,\alpha^\vee \rangle}{|\alpha^\vee|} = \frac{|\alpha^\vee|^2\langle \alpha,y_i\rangle}{2|\alpha^\vee|} = \frac{\langle \alpha,y_i\rangle}{|\alpha|}$. Hence $T_i^\bullet =T_i$. Furthermore,
\begin{align*}
\rho(\tilde{Z}) = \tfrac{1}{2}\sum_{\alpha >0} c_\alpha s_\alpha \otimes \tilde{s}_\alpha &= \tfrac{1}{2}\sum_{\alpha >0, i =1}^n c_\alpha s_\alpha \otimes \frac{B(y_i,\alpha^\vee)}{|\alpha^\vee|}c_i \\
&= \tfrac{1}{2}\sum_{\alpha >0, i =1}^n c_\alpha s_\alpha \otimes \frac{\langle x_i,\alpha^\vee\rangle }{|\alpha^\vee|}c_i = \sum_{i=1}^n T_i \otimes c_i.
\end{align*}

\begin{Proposition}
The element $C_2$ is such that
\[\rho(C_2) = \sum_{i <j} (T_iT_j^\bullet - T_jT_i^\bullet)\otimes c_ic_j + Z_3,
\]
where $Z_3 = \tfrac{1}{4}\sum_{\alpha ,\beta >0}c_\alpha |\alpha^\vee|^{-1} c_\beta |\beta|^{-1} \langle \beta,\alpha^\vee\rangle s_\alpha s_\beta$ is a central element in $\mathbb{C}W$.
\end{Proposition}

\begin{proof}
Note that $\rho(C_2) = \rho\big(\tilde{Z}^2\big) = \big(\sum_{i=1}^n T_i \otimes c_i\big)^2$. Using the super-commutation relations~(\ref{e:CliffRel}) of $c_i$, we obtain that
$\rho(C_2) = \sum_{i<j} [T_i,T_j] \otimes c_ic_j + \sum_{i=1}^n T_i^2$.
Now $T_i^2 = T_iT_i^\bullet = \tfrac{1}{4}\sum_{\alpha ,\beta >0}c_\alpha c_\beta \frac{ \langle x_i,\alpha^\vee\rangle}{|\alpha^\vee|}\frac{\langle \beta,y_i\rangle}{|\beta|} s_\alpha s_\beta$. Therefore,
\[
\sum_{i=1}^n T_i^2 =\tfrac{1}{4}\sum_{i=1}^n \sum_{\alpha ,\beta >0}\frac{c_\alpha}{|\alpha^\vee|} \frac{c_\beta}{|\beta|} \langle x_i,\alpha^\vee\rangle \langle \beta,y_i\rangle s_\alpha s_\beta
= \tfrac{1}{4}\sum_{\alpha ,\beta >0}\frac{c_\alpha}{|\alpha^\vee|} \frac{c_\beta}{|\beta|} \langle \beta,\alpha^\vee\rangle s_\alpha s_\beta ,\]
as required.
\end{proof}

\begin{Corollary}
Let $\mathfrak{D}_{C_2}$ be the Dirac operator as defined in Definition~{\rm \ref{d:Dirac_C}}. Then $\mathfrak{D}_{C_2}$ can be expressed as:
\[\mathfrak{D}_{C_2} = \sum_{i <j} (x_iy_j - x_jy_i) \otimes c_ic_j +\sum_{i<j}( T_iT_j^\bullet - T_jT_i^\bullet) \otimes c_ic_j + Z_3 - \phi.\]
\end{Corollary}

\begin{Remark}
Alternatively, we could slightly modify the generators of $\ama$, writing $\tilde{M}_{ij} = M_{ij} + T_iT_j^\bullet - T_jT_i^\bullet$. Then we can write the Dirac operator as
\[\mathfrak{D}_{C_2} = \sum_{i <j} \tilde{M}_{ij} \otimes c_ic_j + Z_3 - \phi.\]
The expression for $\mathfrak{D}_{C_2}$ above reflects an equivalent definition for the Dirac operator of the degenerate affine Hecke algebra \cite{BCT12}. There
\[D_\mathbb{H} = \sum_{i=1}^n x_i \otimes c_i + \sum_{i=1}^n \mathcal{T}_i \otimes c_i,
\]
where the elements $x_i$ are the commuting generators in the Luzstig presentation of $\mathbb{H}$ and $\mathcal{T}_i =-\frac{1}{2} \sum_{\alpha >0} c_\alpha B^*(x_i,\alpha)s_\alpha$.

However, in the AMA-Dirac we must modify by the central element $Z_3 - \phi \in Z(W)$. This can be interpreted as an analogue of the modification Kostant makes (see~\cite{K03}) when defining the cubic Dirac operator.
\end{Remark}

We will finish by exploring the set of admissible sets in the case when $W=S_n$. Recall that the algebra $\bbc \tilde{W}$ is a $\bbz_2$-graded algebra, where the generators $\tilde{s}_\alpha$ are given odd degree and $\theta$ even degree. The idempotent $\frac{1}{2} (1-\theta)$ is central and $\bbz_2$-homogenous, and hence $\bbc \tilde{W}_- =\frac{1}{2}(1-\theta)\bbc \tilde{W}$ is also $\bbz_2$-graded. Recall also that the simple `pseudo' reflections $\tau_\alpha = \frac{1}{2} (1 - \theta) \tilde{s}_\alpha, \alpha \in \Delta$ generate $\mathbb{C}\tilde{W}_-$.
The decomposition of the group algebra $\bbc \tilde{W} = \bbr \tilde{W} \oplus {\rm i} \bbr \tilde{W}$ is preserved under the multiplication by the idempotent $\frac{1}{2} (1 - \theta)$. Therefore, $\bbc \tilde{W}_- = \bbr \tilde{W}_- \oplus {\rm i} \bbr \tilde{W}_-$.

\begin{Proposition}\label{p::admissible} Let $W$ be the symmetric group~$S_n$. The set of admissible elements in $\mathbb{R}\tilde{W}_-$ is the even centre of $\mathbb{R}\tilde{W}_-$. \end{Proposition}

\begin{proof} Let $\mathcal{A} =\mathcal{A}\big(\bbr\tilde{W}_-\big)$ be the set of admissible elements in $\bbr \tilde{W}_-$. If $C \in \mathcal{A}$ then $C$ is central and fixed by the involution $*$. Since $\tau_\alpha^* = -\tau_\alpha$ in $\bbr\tilde{W}_-$ then any odd central element is not admissible. Hence, $\mathcal{A}$ is contained in $Z_0\big(\bbr\tilde{W}_-\big)$, the even centre of $\bbr\tilde{W}_-$. Now, the even centre
\[Z_0(\bbr\tilde{W}_-)=\big\{\sigma\big(M_1^2,\dots,M_n^2\big)\mid \sigma \textup{ real symmetric polynomial}\big\},
\] where $M_i$ are the Jucys--Murphy elements (see \cite[Lemma~13.1.2, Remark~13.1.3]{K05} and references within). The elements $M_i$ are odd so $M_i^*= -M_i$ and $\big(M_i^2\big)^*=M_i^2\in Z_0\big(\bbr\tilde{W}_-\big)$. Therefore, any expression in the squares of the Jucys--Murphy elements is $*$-invariant and thus, $Z_0\big(\bbr\tilde{W}_-\big)$ is contained in~$\mathcal{A}$.
\end{proof}

\begin{Remark}\label{r:actionofC}
When $W=S_n$, we have shown that the set of admissible elements in $\bbr\tilde{W}_-$ is equal to the even centre $Z_0\big(\bbr\tilde{W}
_-\big)$.
From \cite{K05}, $Z_0\big(\bbr\tilde{W}_-\big)$ is equal to $\big\{\sigma\big(M_1^2,\dots,M_n^2\big)\mid \sigma \textup{ real symmetric polynomial}\big\}$. The element $\sigma\big(M_1^2,\dots,M_n^2\big)$ acts on an irreducible representation by evaluating $\sigma$ at specific real values \cite[Corollary~6.3]{C18}. In particular, we can conclude that for every $\bbc\tilde{W}$-module there is an admissible element that does not act by zero.
\end{Remark}

\begin{Proposition}\label{p:nonzerocoh}
Let $\ama$ be a Dunkl angular momentum algebra with Weyl group $W=S_n$, and let $(\pi,X)$ be a $\bullet$-unitary module for $\ama$, such that the spectrum of $\pi(\Omega_{\mathfrak{sl}(2)})$ is contained in $[-1, + \infty)$. Then, there exists an admissible $C \in \bbc \tilde{W}_-$ such that $H(X,C) \neq 0$.
\end{Proposition}

\begin{proof}
Firstly, it follows from Proposition \ref{p:unitarycoh} that, with our assumptions, $H(X,C) = \ker D_C$, where $D_C = \pi\otimes\sigma(\fD_C)$ and $C$ is an admissible element.

Secondly, as operators on the space $X \otimes S$, the element $\Omega_{\mathfrak{sl}(2)}$ commutes with $\rho(C)$ for every admissible element $C$. Therefore, there exist simultaneous eigenspaces for $\Omega_{\mathfrak{sl}(2)}$ and $\{\rho(C) \mid C \text{ admissible}\}$. Let $U$ be a nonzero eigenspace where $\Omega_{\mathfrak{sl}(2)}$ acts by the scalar $u(\Omega_{\mathfrak{sl}(2)})$ and $\rho(C)$ act by $u(C)$ for $C$ admissible. By Remark \ref{r:actionofC}, there exists an admissible $C$ such that $u(C)\neq 0$. We will show that there is $\lambda \in \bbr\setminus \{0\}$ such that $H(X,C')\neq 0$ for $C' = \lambda C$.

To that end, note that upon changing $C$ to $\big(u(C)^{-1}\sqrt{u(\Omega_{\mathfrak{sl}(2)})+1}\big)C$, we may assume that, as operators on $U$, we have
\[\Omega_{\mathfrak{sl}(2)} = \rho(C)^2 -1.\]
Equation (\ref{e:C-Diracsquare}) states that $ \mathfrak{D}_{C}^2 = \Omega_{\mathfrak{sl}(2)} - \big(\rho(C)^2 - 1\big) + 2\rho(C)\mathfrak{D}_{C}$. Hence, on the eigenspace $U$ we have $ D_{C}^2 -2u(C)D_{C} = 0$ from which we conclude that
\[
\{0\}\neq U \subseteq \ker ((D_{C} -2u(C))\circ D_{C}).
\]
Using the fact that the composition of injective maps is injective, it is not possible that both $\ker D_{C}$ and $\ker (D_{C} -2u(C))$ are equal to
zero. Now using that
\[
\fD_{(-C)} = \fD_0 - \rho(C) = \fD_0 + \rho(C) -2\rho(C) = \fD_C - 2\rho(C),
\]
when restricted to $U$, we obtain $D_{(-C)} = D_{C} -2u(C)$. Therefore, $\ker(D_{C'}) \neq 0$ for some choice of $C' \in \{C,-C\}$.
\end{proof}

\begin{Example}
Let $X_c(\tau)_m \subset M_c(\tau)$ be the harmonic polynomials of degree $m$ as defined in Example~\ref{e:harmonicex}. Then the spectrum of $\Omega_{\mathfrak{sl}(2)}$ on $X_c(\tau)_m$ is contained in $[-1,+\infty)$ for every $m$. Therefore, when $W= S_n$ (with $n \geq 3$), for every $X_c(\tau)_m$ there exists an admissible $C$ such that $H(X_c(\tau)_m, C) \neq 0$.
\end{Example}

\subsection*{Acknowledgements}

This research was supported by Heilbronn Institute for Mathematical Research and the special research fund (BOF) from Ghent University [BOF20/PDO/058].
We would also like to thank Roy Oste for the many discussions while preparing this manuscript and the anonymous referees for their comments and corrections, which greatly improved the manuscript. In particular, we would like to thank them for inspiring us to add Proposition~\ref{p:nonzerocoh} which guarantees that the theory of Dirac operators for the AMA is not a vacuous theory.

\pdfbookmark[1]{References}{ref}
\LastPageEnding

\end{document}